\documentclass[a4paper]{amsart} 

\usepackage[all]{xy}
\usepackage{amsfonts} 
\usepackage{amssymb}
\usepackage{amsmath}
\usepackage{enumerate}
\usepackage{mathbbol} 
\theoremstyle{plain}
\newtheorem{theorem}{Theorem}
\newtheorem{proposition}{Proposition}
\newtheorem{lemma}[proposition]{Lemma}

\theoremstyle{definition}
\newtheorem{definition}[proposition]{Definition}

\theoremstyle{remark}{}
\newtheorem{remark}{Remark}

\newcommand{\secref}[1]{Section~\ref{#1}}
\newcommand{\subsecref}[1]{Subsection~\ref{#1}}
\newcommand{\thmref}[1]{Theorem~\ref{#1}}
\newcommand{\propref}[1]{Proposition~\ref{#1}}
\newcommand{\lemref}[1]{Lemma~\ref{#1}}

\newcommand{\defref}[1]{Definition~\ref{#1}}

\def\R{\mathbb R}

\def\cC{\mathcal C}

\def\cat{\rm cat}

\newcommand{\cof}{\rightarrowtail }

\newdir{ >}{{}*!/-7pt/\dir{>}}

\def\N{{\mathbb{N}}}

\def\R{{\mathbb{R}}}

\def\cat{{\rm{cat}\hskip1pt}}

\def\del{\partial}
\def\Top{{\bf Top}}
\def\id{{\rm id}}
\def\pr{{\rm pr}}
\def\ev{{\rm ev}}
\begin{document}
\title{Simplicial resolutions and Ganea fibrations}
\date{\today}
\author[T. Kahl]{Thomas Kahl}
\address{Centro de Matem{\'a}tica\\
        Universidade do Minho\\
        Campus de Gualtar\\
         4710-057 Braga\\
         Portugal}
\email{kahl@math.uminho.pt}
\author[H. Scheerer]{Hans Scheerer}
\address{Mathematisches Institut\\
        Freie Universit\"at Berlin\\
        Arnimallee 2--6\\
          D--14195  Berlin\\
         Germany}
\email{scheerer@mi.fu-berlin.de}
\author[D. Tanr\'e]{Daniel Tanr\'e}
\address{D\'epartement de Mathematiques\\
         UMR 8524\\
         Universit\'e de Lille~1\\
         59655 Villeneuve d'Ascq Cedex\\
         France}
\email{Daniel.Tanre@agat.univ-lille1.fr}
\author[L. Vandembroucq]{Lucile Vandembroucq}
\address{Centro de Matem{\'a}tica\\
        Universidade do Minho\\
        Campus de Gualtar\\
         4710-057 Braga\\
         Portugal}
\email{lucile@math.uminho.pt}
\begin{abstract}
In this work, we compare the two approximations of a path-connected space $X$, by
the Ganea spaces $G_n(X)$ and by the realizations $\|\Lambda_\bullet X\|_{n}$ of the truncated simplicial resolutions emerging from the loop-suspen\-sion co\-triple $\Sigma\Omega$. 
For a simply connected space $X$, we construct maps\\ $\|\Lambda_\bullet X\|_{n-1}\to G_n(X)\to \|\Lambda_\bullet X\|_{n}$ over $X$, up to homotopy. 
In the case $n=2$,   we prove the existence of a map
$G_2(X)\to\|\Lambda_\bullet X\|_{1}$ over $X$ (up to homotopy) and conjecture that this map exists for any $n$. 
\end{abstract}
\maketitle
We use  the category $\Top$ of well pointed compactly generated spaces having the homotopy type of CW-complexes.  We denote by $\Omega$ and $\Sigma$ the classical loop space and (reduced) suspension constructions on $\Top $.

\medskip
Let $X\in \Top$. First we recall the construction of the Ganea fibrations $G_n(X)\to X$ where $G_n(X)$ has the same homotopy type as the $n$-th stage, $B_n\Omega X$,  of the construction of the  classifying space of $\Omega X$:
\begin{enumerate}
\item  the first Ganea fibration, $p_1\colon G_1(X)\rightarrow X$,  is the associated fibration to the evaluation map
$\ev_X\colon \Sigma\Omega X\rightarrow X$;
\item given the $n${th}-fibration $p_n\colon G_n(X)\rightarrow X$, let  $F_n(X)$ be its homotopy fiber and
let  $G_n(X)\cup {\cC}(F_n(X))$ be the mapping cone of the inclusion $F_n(X)\rightarrow G_n(X)$. We define now a map ${p'}_{n+1}\colon G_n(X)\cup {\cC}(F_n(X))\rightarrow X$ as
$p_n$ on $G_n(X)$ and that sends the (reduced) cone $\cC(F_n(X))$ on the base point. The
$(n+1)$-{{st}}-fibration of Ganea, $p_{n+1}\colon G_{n+1}(X)\rightarrow X$,
is the fibration associated to ${p'}_{n+1}$.
\item Denote by $G_\infty(X)$ the direct limit of the canonical maps $G_n(X)\to G_{n+1}(X)$ and by $p_\infty\colon G_\infty(X)\to X$ the map induced by the $p_n$'s.
\end{enumerate}

\smallskip
From a classical  theorem of Ganea \cite{Gan67a}, one knows that the fiber of $p_n$ has the homotopy type of an $(n+1)$-fold reduced join of $\Omega X$ with itself. Therefore the maps $p_n$ are higher and higher connected when the integer $n$ grows. As a consequence, if $X$ is path-connected, the map 
$p_\infty\colon G_\infty(X)\to X$ is a homotopy equivalence and  the total spaces $G_n(X)$ constitute approximations of the space $X$.

\medskip
The previous construction starts with the couple of adjoint functors  $\Omega$ and $\Sigma$. 
From them, we can construct a \emph{simplicial space} $\Lambda_{\bullet} X$, defined by $\Lambda_nX=(\Sigma\Omega )^{n+1}X$ and augmented by 
$d_0=\ev_X\colon \Sigma\Omega X\rightarrow X$. 
Forgetting the degeneracies, we have a \emph{facial space} (also called {restricted simplicial space} in \cite[3.13]{DD77}).
Denote by $\|\Lambda_{\bullet} X\|$ the realization  of this facial space (see \cite{Seg74} or \secref{sec:facial}). 
An adaptation of the proof of Stover (see \cite[Proposition 3.5]{Sto90}) shows that the augmentation $d_0$ induces a map
$\|\Lambda_\bullet X\|\to X$ which is a homotopy equivalence.
If we consider the successive stages of the realization of the facial space $\Lambda_\bullet X$,   we get maps 
$\|\Lambda_\bullet X\|_n\to X$ which constitute a second sequence of approximations of the space $X$. In this work, we study the relationship between these two sequences of approximations and prove the following results.

\begin{theorem}\label{thm:main}
Let $X\in \Top$ be a simply connected space. Then there is  a homotopy commutative diagram
$$\xymatrix{
\|\Lambda_\bullet X\|_{n-1}\ar[r]\ar[ddr]&
G_n(X)\ar[r]\ar[dd]_{p_n}&
\|\Lambda_\bullet X\|_{n}\ar[ldd]\\
&&\\
&X&
}$$
\end{theorem}

The hypothesis of simply connectivity is used only for the map $G_n(X)\to \|\Lambda_\bullet X\|_n$, see \thmref{thm:easyway} and \thmref{thm:hardway}.
In the case $n=2$, the situation is better.

\begin{theorem}\label{thm:main2}
Let $X\in\Top$. Then there are   homotopy commutative triangles
$$\xymatrix{
\|\Lambda_\bullet X\|_{1}\ar@<1ex>[rr]\ar[dr]&&
G_2(X)\ar[ld]^{p_2}\ar[ll]\\
&X&
}$$
\end{theorem}

\medskip
We conjecture the existence of maps
$\xymatrix@1{\|\Lambda_{\bullet} X\|_{n-1}\ar@<1ex>[r]&
G_n(X)\ar[l]
}$ over $X$ up to homotopy, for any $n$.

This work may also be seen as a comparison of two constructions: an iterative fiber-cofiber  process and the realization of progressive truncatures of a facial resolution. More generally, for any  cotriple, we present an adapted fiber-cofiber construction (see \defref{defi:TGanea}) and ask if the results obtained in the case of $\Sigma\Omega$ can be extended to this setting.

Finally, we observe that a variation on a theorem of Libman is essential in our argumentation, see \thmref{Libman}. A proof of this result, inspired by the methods developed by R. Vogt (see \cite{Vogt73}), is presented in an Appendix.

\smallskip
This program is carried out in Sections 1-8 below, whose headings are
self-explanatory:

\tableofcontents
\section{Facial spaces}\label{sec:facial}
%

A \emph{facial object} in a category $\bf C$ is a sequence of
objects $X_0, X_1, X_2, \dots$ together with morphisms $d_i: X_n \to
X_{n-1}$, $0 \leq i \leq n$, satisfying the \emph{facial
identities} $d_id_j = d_{j-1}d_i$ $(i < j)$.
$$\xymatrix{
X_{0} &\ar@<2pt>[l]^{d_1} \ar@<-2pt>[l]_{d_0}X_{1}
&\ar@<4pt>[l]^{d_2} \ar[l]|{d_1}\ar@<-4pt>[l]_{d_0}X_{2} &\cdots
&X_{n-1} &\ar@<4pt>[l]^-{d_n}
\ar@{}[l]|-{:}\ar@<-5pt>[l]_-{d_0}X_{n} &\ar@<4pt>[l]
\ar@{}[l]|-{:}\ar@<-5pt>[l]\cdots
 }$$
 The morphisms $d_i$ are
called \emph{face operators}. We shall use notation like
$X_{\bullet}$ to denote facial objects. With the obvious morphisms
the facial objects in $\bf C$ form a category which we denote by
$d{\bf C}$. An \emph{augmentation} of a facial object
$X_{\bullet}$ in a category $\bf C$ is a morphism $d_0: X_0 \to X$
with $d_0 \circ d_0 = d_0 \circ d_1$. The facial object
$X_{\bullet}$ together with the augmentation $d_0$ is called a \emph{facial
resolution of $X$} and is denoted by $X_{\bullet}
\stackrel{d_0}{\to} X$.\\

\subsection{Realization(s) of a facial space}

As usual, $\Delta ^n$ denotes the standard $n$-simplex of
$\R^{n+1}$ and the inclusions of faces are denoted by
$\delta^i:\Delta ^n\to \Delta^{n+1}$. We consider the point
$(0,\dots, 0,1) \in \R^{n+1}$ as the base-point of the standard
$n$-simplex $\Delta ^n$. If $X$ and $Y$ are in $\Top$, we denote by $X\rtimes Y$ the half smashed product $X\rtimes Y =X\times Y/*\times Y$.\\

\noindent A \emph{facial space} is a facial object in $\Top$. The \emph{realization} of a facial space $X_{\bullet}$ is the
direct limit
$$\|X_{\bullet}\|_{\infty} = \lim \limits_{\longrightarrow} \|X_{\bullet}\|_n$$
where the spaces $\|X_{\bullet}\|_n$ are inductively defined as follows.
Set $\|X_{\bullet}\|_0 = X_0$. Suppose we have defined
$\|X_{\bullet}\|_{n-1}$ and a map $\chi_{n-1}: X_{n-1}\rtimes \Delta^{n-1} \to \|X_{\bullet}\|_{n-1}$
($\chi_0$ is the obvious homeomorphism). Then $\|X_{\bullet}\|_{n}$ and $\chi_n$ are defined by the pushout diagram
$$\xymatrix{
X_n\rtimes \partial \Delta ^n\ar[r]^{\varphi_n}_{} \ar@{
>->}[d]^{}_{} & \|X_{\bullet}\|_{n-1}\ar@{ >->}[d]^{}_{}
\\
X_n\rtimes \Delta ^n\ar[r]^{}_{\chi_n} & \|X_{\bullet}\|_{n} }$$
where $\varphi_n$ is defined by the following requirements, for any $i\in\left\{0,1,\ldots,n\right\}$, 
$$\varphi_n\circ (X_n\rtimes \delta^i) = \chi_{n-1}\circ (d_i \rtimes \Delta ^{n-1}):
X_n\rtimes \Delta^{n-1} \to \|X_{\bullet}\|_{n-1}.$$ It is clear that $\varphi_1$ is a well-defined continuous map.
For $\varphi_n$ with $n\geq 2$, this is assured by the facial identities
$d_id_j = d_{j-1}d_i\; (i < j)$.

We  also consider another realization of the facial space
$X_{\bullet}$. The \emph{free realization} of $X_{\bullet}$ is the
direct limit
$$|X_{\bullet}|_{\infty} = \lim \limits_{\longrightarrow} |X_{\bullet}|_n$$
where the spaces $|X_{\bullet}|_n$ are inductively defined as
follows. Set $|X_{\bullet}|_0 = X_0$. Suppose we have defined
$|X_{\bullet}|_{n-1}$ and a map $\bar \chi_{n-1}: X_{n-1}\times
\Delta^{n-1} \to |X_{\bullet}|_{n-1}$ ($\bar \chi_0$ is the obvious
homeomorphism). Then $|X_{\bullet}|_{n}$ and $\bar \chi_n$ are
defined by the pushout diagram
$$\xymatrix{
X_n\times \partial \Delta ^n\ar[r]^{\bar \varphi_n}_{} \ar@{
>->}[d]^{}_{} & |X_{\bullet}|_{n-1}\ar@{ >->}[d]^{}_{}
\\
X_n\times \Delta ^n\ar[r]^{}_{\bar\chi_n} & |X_{\bullet}|_{n} }$$
where $\bar \varphi_n$ is defined by the following requirements, for any $i\in\left\{0,1,\ldots,n\right\}$, 
$$\bar \varphi_n\circ (X_n\times \delta^i) = \bar \chi_{n-1}\circ (d_i \times \Delta ^{n-1}):
X_n\times \Delta^{n-1} \to |X_{\bullet}|_{n-1}.$$ Again the facial identities $d_id_j = d_{j-1}d_i\; (i < j)$
assure that $\bar \varphi_n$ is a well-defined continuous map. Since
$\bar\chi_{n-1}$ is base-point preserving, so is $\bar \varphi_n$ and
hence $\bar\chi_n$.\\

We  sometimes consider facial spaces with upper indexes
$X^{\bullet}$. In such a case, the realizations up to $n$ are
denoted by $||X^{\bullet}||^n$ and $|X^{\bullet}|^n$.\\

Let $X_{\bullet} \stackrel{d_0}{\to} X$ be a facial resolution of  a
space $X$. We define a sequence of maps $\|X_{\bullet}\|_n \to X$ as
follows. The map $\|X_{\bullet}\|_0 \to X$ is the augmentation.
Suppose we have defined $\|X_{\bullet}\|_{n-1} \to X$ such that the
following diagram is commutative:
$$\xymatrix{
X_{n-1}\rtimes \Delta ^{n-1} \ar[r]^-{\chi_{n-1}}_{} \ar[d]^{}_{\pr}
& \|X_{\bullet}\|_{n-1} \ar [d]^{}_{}
\\
X_{n-1} \ar[r]^{}_{(d_0)^n}
& X,
}$$
where $(d_0)^n$ denotes the $n$-fold composition of the face operator $d_0$.
Consider the diagram
$$\xymatrix{
X_{n}\rtimes \Delta ^{n-1} \ar[r]^{d_i\rtimes \Delta^{n-1}}_{} \ar[d]^{}_{X_n\rtimes \delta ^i}
& X_{n-1}\rtimes \Delta ^{n-1} \ar [d]^{\chi_{n-1}}
\\
X_{n}\rtimes \partial \Delta ^{n} \ar[r]^{\varphi_{n}}_{}
\ar[d]^{}_{\pr} & \|X_{\bullet}\|_{n-1} \ar [d]^{}_{}
\\
X_{n} \ar[r]^{}_{(d_0)^{n+1}}
& X.
}$$
The upper square is commutative for all $i$ and so is the outer diagram.
It follows that the lower square is commutative.
We may therefore define $\|X_{\bullet}\|_n \to X$ to be the unique map which extends $\|X_{\bullet}\|_{n-1} \to X$
and which, pre-composed by $\chi _n$, is the composite
$\xymatrix@1{X_n\rtimes \Delta ^n\ar[r]^-{\pr}&X_n\ar[rr]^{(d_0)^{n+1}}&&X}$.
Similarly, we define a sequence of maps $|X_{\bullet}|_n \to X$.
We refer to the maps $\|X_{\bullet}\|_n \to X$ and $|X_{\bullet}|_n \to X$ as the \emph{canonical maps}
induced by the facial resolution $X_{\bullet} \to X$. The  next statement relates these two realizations; its proof is straightforward. 

\begin{proposition}\label{pointedvsfree}
Let $X_{\bullet}$ be a facial space. Then for each $n\in \N$, the canonical map $|X_{\bullet}|_n \to X$
factors through the canonical map $\|X_{\bullet}\|_n \to X$
\end{proposition}


\subsection{Facial resolutions with contraction}

A \emph{contraction} of a facial resolution $X_{\bullet} \stackrel{d_0}{\to} X$
consists of a sequence of morphisms $s: X_{n-1} \to X_n \quad (X_{-1} = X)$
such that $d_0 \circ s = \id$ and $d_i\circ s = s\circ d_{i-1}$ for $i\geq 1$.

\begin{proposition}\label{quotient}
Let $X_{\bullet}\stackrel{d_0}{\to} X$ be a facial resolution which
admits a contraction $s: X_{n-1} \to X_n \quad (X_{-1} = X)$.
For any $n\geq 0$, $|X_{\bullet}|_{n}$ can be identified with the
quotient space $X_n\times \Delta^n /\sim$
 where the relation is given by
 $$(x,t_0,...,t_k,...,t_n)\sim (sd_kx,0,t_0,...,\hat{t}_k,...,t_n),\quad \text{if  } t_k=0.$$
As usual, the expression $\hat{t}_k$ means that $t_k$ is omitted.
 Under this identification the canonical map $|X_{\bullet}|_{n}\to
 X$ is given by $[x,t_0,...,t_k,...,t_n]\mapsto (d_0)^{n+1}(x)$  and the inclusion
 $|X_{\bullet}|_{n}\cof
 |X_{\bullet}|_{n+1}$ is given by $[x,t_0,...,t_k,...,t_n]\mapsto
 [sx,0,t_0,...,t_k,...,t_n]$.
\end{proposition}

\begin{proof} We first note that the simplicial identities together
with the contraction properties guarantee that the relation is
unambiguously defined if various parameters are zero and also that
the two maps
$$\begin{array}{rcl}
X_n\times \Delta^n /\sim & \to & X_{n+1}\times \Delta^{n+1}/\sim\\
\left[x,t_0,...,t_k,...,t_n\right]& \mapsto
&\left[sx,0,t_0,...,t_k,...,t_n\right]
\end{array}$$
and
$$\begin{array}{rcl}
X_n\times \Delta^n /\sim & \to & X\\
\left[x,t_0,...,t_k,...,t_n\right]& \mapsto & (d_0)^{n+1}(x)
\end{array}$$
that we will denote by $\iota_n$ and $\varepsilon_n$ respectively
are well-defined.

Beginning with $\xi_0=\id$, we next construct a sequence of
homeomorphisms
 $\xi_n:|X_{\bullet}|_{n} \to X_n\times \Delta^n /\sim$
 inductively by using the universal property of pushouts in the diagram
$$\xymatrix{
X_n\times \partial \Delta ^n \ar[r]^{\bar\varphi_n}_{} \ar@{
>->}[dd]^{}_{}
& |X_{\bullet}|_{n-1}\ar@{ >->}[dd]^{}_{} \ar[rd]^{\xi_{n-1}}\\
&&X_{n-1}\times \Delta^{n-1} /\sim \ar[dd]^{\iota_{n-1}}_{}
\\
X_n\times \Delta ^n\ar[r]^{}^{\bar\chi_n} \ar[rrd]_{q_n}
& |X_{\bullet}|_{n}\ar@{.>}[dr]^{\xi_n} &\\
&&X_n\times \Delta^n /\sim }$$ where $q_n$ is the identification
map. If $t_k=0$, the construction up to $n-1$ implies 
$$\xi_{n-1}\circ
\bar\varphi_n(x,t_0,...,t_n)=q_{n-1}\circ (d_k\times
\Delta^{n-1})=[d_kx,t_0,...\hat{t}_k,...,t_n].$$
Therefore, we see that the diagram
$$\xymatrix{
X_n\times \partial \Delta ^n\ar[rr]^-{\xi_{n-1}\circ\bar\varphi_n}_{}
\ar@{
>->}[d]^{}_{} && X_{n-1}\times \Delta^{n-1} /\sim
\ar[d]^{\iota_{n-1}}_{}
\\
X_n\times \Delta ^n\ar[rr]^{}_{q_n} && X_n\times \Delta^n /\sim }$$
is commutative and, by checking the universal property, that it is a
pushout. Thus $\xi_n$ exists and is a homeomorphism. Through this
sequence of homeomorphisms, $\iota_n$ corresponds to the inclusion
$|X_{\bullet}|_{n}\cof |X_{\bullet}|_{n+1}$ and $\varepsilon_n$ to
the canonical map $|X_{\bullet}|_{n}\to X$.
\end{proof}

\begin{proposition}\label{facialres}
Let $X_{\bullet}\stackrel{d_0}{\to} X$ be a facial resolution which
admits a natural contraction $s: X_{n-1} \to X_n \quad (X_{-1} = X)$.
For any $n\geq 0$, the canonical map $|X_{\bullet}|_n\to X$ admits a
(natural) section $\sigma_n:X\to |X_{\bullet}|_n$ and the inclusion
$|X_{\bullet}|_{n-1}\cof |X_{\bullet}|_{n}$ is naturally homotopic
to $\sigma_n$ pre-composed by the canonical map:
$$\xymatrix{
|X_{\bullet}|_{n-1}\ar[rr]\ar[rd] && |X_{\bullet}|_n\\
& X \ar[ur]_{\sigma_n}& }$$ \noindent In particular, if the facial
resolution $X_{\bullet} \to \ast$ admits a natural contraction then the
inclusions $|X_{\bullet}|_{n-1} \cof |X_{\bullet}|_{n}$ are
naturally homotopically trivial.
\end{proposition}

\begin{proof}
Through the identification established in
\propref{quotient}, the section $\sigma_n:X\to |X_{\bullet}|_n$ is given
by
$$\sigma_n(x)=[(s)^{n+1}(x),0,...,0,1].$$
 Using the fact that
$$sd_nsd_{n-1}\cdots sd_2sd_1 s= (s)^{n+1}(d_0)^{n},$$
we calculate that the (well-defined) map
$H:|X_{\bullet}|_{n-1}\times I\to |X_{\bullet}|_{n-1}$ given by
$$H([x,t_0,...,t_{n-1}],u)=[sx,u,(1-u)t_0,...,(1-u)t_{n-1}]$$
is a homotopy between the inclusion and $\sigma_n$ pre-composed by
the canonical map $|X_{\bullet}|_{n-1}\to X$.
\end{proof}


\section{First part of \thmref{thm:main}:  the map $\|\Lambda_{\bullet} X\|_{n-1} \to G_n(X)$}

Let $X\in\Top$. We consider the facial resolution $\Lambda_{\bullet}(X) \to X$ where $\Lambda_{n}(X) = (\Sigma \Omega)^{n+1} X$, the face operators $d_i : (\Sigma \Omega)^{n+1} X \to (\Sigma \Omega)^n X$ are defined by $d_i = (\Sigma \Omega )^i(\ev_{(\Sigma \Omega)^{n-i}X})$, and the augmentation is $d_0=\ev_X : \Sigma \Omega X \to X$. 

\begin{theorem}\label{thm:easyway}
Let $X\in\Top$. For each $n\in\N$, the canonical map
$\|\Lambda_\bullet X\|_{n-1}\to X$
factors through the Ganea fibration 
$G_n(X)\to X$.
\end{theorem}

The proof  uses the next result.

\begin{lemma}\label{lem:Puppetrick}
Given a pushout
$$\xymatrix{
\Sigma A\rtimes \partial \Delta^n\ar[rr]\ar@{ >->}[d]^{}_{}&&Y\ar[d]^f\\
\Sigma A\rtimes \Delta^n\ar[rr]&&Y'\\}$$
where the left-hand vertical arrow is a cofibration, 
then there exists a cofiber sequence
$\xymatrix@1{\Sigma A\land \partial \Delta^n\ar[r]&Y\ar[r]^f&Y'}$.
\end{lemma}

\begin{proof}
With the Puppe trick, we construct a commutative diagram
$$\xymatrix{
\Sigma A\vee \left(\Sigma A\wedge \partial \Delta^n\right)\ar@{ >->}[d]&&
\ar[ll]_-{\sim}\left(\Sigma A\rtimes \partial \Delta^n\right)\ar[d]\\
\Sigma A\vee \left(\Sigma A\wedge  \Delta^n\right)&&
\ar[ll]^-{\sim}\left(\Sigma A\rtimes  \Delta^n\right)\\}$$
from which we obtain a commutative diagram
$$\xymatrix{
\Sigma A\vee \left(\Sigma A\wedge \partial \Delta^n\right)\ar@{ >->}[d]\ar[rr]^-{\sim}&&
\left(\Sigma A\rtimes \partial \Delta^n\right)\ar[d]\\
\Sigma A\vee \left(\Sigma A\wedge  \Delta^n\right)\ar[rr]_-{\sim}&&
\left(\Sigma A\rtimes  \Delta^n\right)\\}$$
because the left-hand vertical arrow is a cofibration. We form now
$$\xymatrix{
\Sigma A\wedge \partial \Delta^n\ar[r]\ar@{ >->}[dd]\ar[r]&
\Sigma A\vee \left(\Sigma A\wedge \partial \Delta^n\right)\ar@{ >->}[dd]\ar[rr]^{\sim}&&
\Sigma A\rtimes \partial \Delta^n\ar[ld]\ar[rr]\ar[dd]&&Y\ar[dl]\ar[dd]\\
&&\bullet_1\ar[rr]\ar[rd]^{\sim}&&\bullet_2\ar[rd]^{\sim}&\\
\Sigma A\wedge  \Delta^n\ar[r]&\Sigma A\vee \left(\Sigma A\wedge  \Delta^n\right)\ar[ru]^{\sim}\ar[rr]^{\sim}&&
\Sigma A\rtimes \Delta^n\ar[rr]&&Y'\\}$$
where $\bullet_1$ and $\bullet_2$ are built by pushout and the left-hand square is a pushout. The map $\bullet_2\rightarrow Y'$ is a weak equivalence because it is induced between pushouts by the weak equivalence
$\bullet_1\rightarrow \Sigma A\rtimes \Delta^n$.
\end{proof}

\begin{proof}[Proof of \thmref{thm:easyway}]
We suppose that $\Phi_{n-2}\colon \|\Lambda_\bullet X\|_{n-2}\to G_{n-1}(X)$ has been constructed over $X$ and observe that the existence of $\Phi_0$ is immediate. We consider the following commutative diagram
$$\xymatrix{
(\Sigma\Omega)^{n}(X)\wedge \partial \Delta^{n-1}\ar@{-->}[rr]^-{\hat{\Phi}_{n-2}}\ar[d]_{\tilde{v}_{n-2}}
&&F_{n-1}(X)\ar[d]\\
\|\Lambda_\bullet X\|_{n-2}\ar[rr]^{\Phi_{n-2}}\ar[d]_{v_{n-2}}
\ar@/^1pc/[ddr]^-{\lambda_{n-2}}
&&G_{n-1}(X)\ar[ddl]^{p_{n-1}}\\
\|\Lambda_\bullet X\|_{n-1} \ar@/_1pc/[rd]_{\lambda_{n-1}}&&\\
&X&\\}$$
where the left-hand column is a cofibration sequence by \lemref{lem:Puppetrick}. From the equalities
\begin{eqnarray*}
p_{n-1}\circ \Phi_{n-2}\circ \tilde{v}_{n-2}&=&
\lambda_{n-2}\circ  \tilde{v}_{n-2}\\
&=& \lambda_{n-1}\circ v_{n-2}\circ  \tilde{v}_{n-2}\simeq\ast,
\end{eqnarray*}
we deduce a map $\hat{\Phi}_{n-2}\colon (\Sigma\Omega)^{n}(X)\land \partial \Delta^{n-1}\rightarrow F_{n-1}(X)$ making the diagram homotopy commutative.
From the definition of $G_{n}(X)$ as a cofiber,
this gives a map $\Phi_{n-1}\colon \|\Lambda_\bullet X\|_{n-1}\rightarrow G_{n}(X)$ over $X$.
\end{proof}
Instead of the explicit construction above, we can also observe that the cone length of  $ \|\Lambda_\bullet X\|_{n-1}$ is less than or equal to $n$ and deduce \thmref{thm:easyway} from  basic results on Lusternik-Schnirelmann category, see \cite{CLOT}.

\section{The facial space ${\mathcal G} _{\bullet} (X)$}

For a space $X$ we denote by $P'X$ the Moore path space and by
$\Omega 'X$ the Moore loop space. Path multiplication turns $\Omega
'X$ into a topological monoid. Given a space $X$, we define the
facial space ${\mathcal G} _{\bullet} (X)$ by ${\mathcal G} _{n} (X)
= (\Omega 'X)^{n}$ with the face operators $d_i : (\Omega 'X)^{n}
\to (\Omega 'X)^{n-1}$ given by
$$d_i(\alpha_1,...,\alpha_n)=\left\{
\begin{array}{lc}
(\alpha_2,...,\alpha_n) & i=0\\
(\alpha_1,...,\alpha_{i-1},\alpha_i\alpha_{i+1},...,\alpha_n) & 0<i<n\\
(\alpha_1,...,\alpha_{n-1}) & i=n.\\
\end{array}\right. $$
\emph{The purpose of this section is to compare the free realization of ${\mathcal G} _{\bullet} (X)$ to the construction of the classifying space of $\Omega 'X$.}

 We work with the following construction of $B\Omega 'X$. The classifying space $B\Omega' X$ is the orbit space of the contractible $\Omega' X$-space $E\Omega 'X$ which is obtained as the direct limit of a
sequence of $\Omega' X$-equivariant cofibrations $E_n\Omega 'X \cof E_{n+1}\Omega 'X$. The spaces
$E_n\Omega 'X$ are inductively defined by $E_0\Omega 'X = \Omega 'X$, $E_{n+1}\Omega 'X = E_n\Omega 'X
\cup_{\theta} (\Omega 'X\times CE_n\Omega 'X)$ where $\theta$ is the action
$\Omega 'X\times E_n\Omega 'X \to E_n\Omega 'X$ and $C$ denotes the free (non-reduced) cone construction. The orbit spaces of the $\Omega 'X$-spaces $E_n\Omega ' X$ are denoted by $B_n\Omega 'X$.
For each $n\in \N$ this construction yields a cofibration $B_n\Omega 'X \cof B\Omega 'X$. It is well known that for simply connected spaces this cofibration is equivalent to the $n$th Ganea map $G_n(X) \to X$.

\begin{proposition}   \label{cat}
For each $n \in \N$ there is a natural commutative diagram
$$\xymatrix{
B_n\Omega 'X\ar[r]^{}_{} \ar[d]^{}_{}
& |{\mathcal G} _{\bullet} (X)|_n\ar [d]^{}_{}
\\
B\Omega 'X\ar[r]^{}_{}
& |{\mathcal G} _{\bullet} (X)|_{\infty}
}$$
in which the bottom horizontal map is a homotopy equivalence. 
\end{proposition}

\begin{proof}
We  obtain the diagram of the statement from a diagram of $\Omega 'X$-spaces by passing to orbit spaces. Consider the facial $\Omega'X$-space $P_{\bullet}(X)$ in which $P _{n} (X)$ is the
free $\Omega'X$-space $\Omega' X\times (\Omega 'X)^{n}$ and the
face operators $d_i : (\Omega 'X)^{n+1} \to (\Omega 'X)^{n}$ (which
are equivariant) are given by
$$d_i(\alpha_0,...,\alpha_n)=\left\{
\begin{array}{lc}
(\alpha_0,...,\alpha_{i-1},\alpha_i\alpha_{i+1},...,\alpha_n) & 0\leq i<n\\
(\alpha_0,...,\alpha_{n-1}) & i=n.\\
\end{array}\right.$$
The maps $s\colon P_{n-1}(X)\to P_n(X)$ given by 
$s(\alpha_0,\ldots,\alpha_{n-1})=(\ast,\alpha_0,\ldots,\alpha_{n-1})$
constitute a natural contraction of the facial resolution
$P_\bullet(X)\to\ast$. By \propref{facialres}, the maps
$|P_\bullet(X)|_{n-1}\to |P_\bullet(X)|_n$
are hence naturally homotopically trivial.

The construction of the realization of $P_{\bullet}(X)$ yields
$\Omega'X$-spaces. We construct a natural commutative diagram of equivariant maps
$$\xymatrix{
E_0\Omega 'X \ar@{ >->}[r]^{}_{} \ar[d]_{g_0}_{}
& E_1\Omega 'X\ar@{ >->}[r]^{}_{} \ar[d]^{g_1}_{} & \cdots \ar@{ >->}[r]^{}_{}
& E_n\Omega 'X\ar@{ >->}[r]^{}_{} \ar[d]^{g_n}_{} & \cdots
\\
|P_{\bullet}(X)|_0 \ar@{ >->}[r]^{}_{}
& |P_{\bullet}(X)|_1 \ar@{ >->}[r]^{}_{} & \cdots \ar@{ >->}[r]^{}_{}
& |P_{\bullet}(X)|_n \ar@{ >->}[r]^{}_{} & \cdots 
}$$
inductively as follows: The map $g_0$ is the identity $\Omega 'X \stackrel{=}{\to} \Omega 'X$. Suppose that $g_n$ is defined. Since the map $|P_{\bullet}(X)|_n \cof |P_{\bullet}(X)|_{n+1}$ is naturally homotopically trivial, it factors naturally through the cone $C|P_{\bullet}(X)|_n$. Extend this factorization equivariantly to obtain the following commutative diagram of $\Omega 'X$-spaces:
$$\xymatrix{
\Omega 'X\times |P_{\bullet}(X)|_n\ar[r]^{}_{} \ar[d]^{}_{}
& |P_{\bullet}(X)|_n\ar [d]^{}_{}
\\
\Omega 'X\times C|P_{\bullet}(X)|_n\ar[r]^{}_{}
& |P_{\bullet}(X)|_{n+1}.
}$$
Define $g_{n+1}$ to be the composite
\begin{eqnarray*}
\lefteqn{E_n\Omega 'X\cup_{\Omega 'X\times E_n\Omega 'X}(\Omega 'X\times CE_n\Omega 'X)}\\ 
&\to& |P_{\bullet}(X)|_n\cup_{\Omega 'X\times |P_{\bullet}(X)|_n}(\Omega 'X\times C|P_{\bullet}(X)|_n)\\
&\to& |P_{\bullet}(X)|_{n+1}.
\end{eqnarray*}
It is clear that $g_{n+1}$ is natural. In the direct limit we obtain a natural equivariant map $g : E\Omega 'X \to |P_{\bullet}(X)|_{\infty}$. This map is a homotopy equivalence. Indeed, $E\Omega 'X$ is contractible and, since each inclusion $|P_{\bullet}(X)|_n \cof |P_{\bullet}(X)|_{n+1}$ is homotopically trivial, $|P_{\bullet}(X)|_{\infty}$ is contractible, too. For each $n \in \N$ we therefore obtain the following natural commutative diagram of $\Omega'X$-spaces:
$$\xymatrix{
E_n\Omega 'X\ar[r]^{}_{} \ar[d]^{}_{}
& |P _{\bullet} (X)|_n\ar [d]^{}_{}
\\
E\Omega 'X\ar[r]^{\sim}_{}
& |P _{\bullet} (X)|_{\infty}.
}$$
Passing to the orbit spaces, we obtain the diagram of the statement. It follows for instance from \cite[1.16]{LSabstract} that the map $B\Omega 'X \to |{\mathcal G} _{\bullet} (X)|_{\infty}$ is a homotopy equivalence.
\end{proof}

\begin{remark} 
Note that the upper horizontal map in the diagram of \propref{cat} is not a homotopy equivalence in general. Indeed, for $X = *$, $B_1\Omega'X$ is contractible but $|{\mathcal G}_{\bullet}(X)|_1 \simeq S^1$. It can, however, be shown that there also exists a diagram as in \propref{cat} with the horizontal maps reversed.
\end{remark}

\section{The facial resolution $\Omega'\Lambda_{\bullet}X \to \Omega' X$ admits a  contraction}
Consider the natural map $\gamma _X \colon \Omega 'X \to \Omega' \Sigma \Omega X$, $\gamma _X(\omega ,t) = (\nu (\omega ,t),t)$ where $\nu (\omega ,t): \R^+ \to \Sigma \Omega X$ is given by
$$\nu (\omega ,t)(u) = \left\{ \begin{array}{ll}
\left[\omega _t, \frac{u}{t} \right], &  u < t,\\
\left[c_*,0\right], &   u \geq t.
\end{array}\right.$$
Here, $c_*$ is the constant path $u \mapsto *$ and $\omega_t \colon I \to X$ is the loop defined by $\omega_t(s) = \omega (ts)$.

\begin{lemma}
The map $\gamma_X$ is continuous.
\end{lemma}

\begin{proof}
It suffices to show that the map $\nu ^{\flat} : \Omega 'X \times \R^+ \to \Sigma \Omega X$, $(\omega ,t,u) \mapsto \nu(\omega ,t)(u)$ is continuous. Consider the subspace $W = \{\omega \in X^{\R^+}: \omega (0) = \ast\}$ of $X^{\R^+}$ and the continuous map $\rho : W \times \R^+ \to X^{\R^+}$ given by
$$\rho (\omega ,t)(u) = \left\{ \begin{array}{ll}
\omega (u), &  u\leq t,\\
\omega (t), &  u\geq t.
\end{array}\right.$$
Note that if $(\omega ,t) \in P'X$ then $\rho(\omega ,t) = \omega$. Consider the continuous map $$\phi : W \times \R^+ \times [0,\frac{\pi}{2}] \to \Sigma P'X$$ defined by
$$\phi(\omega ,r, \theta) = \left\{ \begin{array}{ll}
\left[\rho(\omega, r\cos \theta), r\cos \theta, \tan \theta \right], &  \theta \leq \frac{\pi}{4},\\
\left[c_*, 0, 0 \right], & \theta \geq \frac{\pi}{4}.
\end{array}\right.$$
When $r = 0$, we have $\phi(\omega ,r, \theta) = [c_*,0,0]$ for any $\theta$. Therefore $\phi$ factors through the identification map
$$W\times \R^+ \times [0,\frac{\pi}{2}] \to W\times \R^+ \times \R^+, (\omega ,r,\theta) \mapsto (\omega ,r\cos \theta, r\sin \theta)$$
and induces a continuous map $\psi : W \times \R^+ \times \R^+ \to \Sigma P'X$. Explicitly,
$$\psi(\omega ,t, u) = \left\{ \begin{array}{ll}
\left[\rho(\omega, t), t, \frac{u}{t} \right], &  u < t,\\
\left[c_*, 0, 0 \right], & u \geq t.
\end{array}\right.$$
Consider the continuous map $\xi : P'X \to PX$ defined by $\xi(\omega ,t)(s) = \omega (ts)$. Note that $\xi (\omega, t) = \omega _t$ if $(\omega,t) \in \Omega'X$ and, in particular, that $\xi (c_*,0) = c_*$. The restriction of $\Sigma \xi \circ \psi$ to $\Omega 'X \times \R^+$ factors through the subspace $\Sigma \Omega X$ of $\Sigma PX$ and the continuous map $$\Omega 'X \times \R^+ \to \Sigma \Omega X, (\omega ,t, u) \mapsto (\Sigma \xi \circ \psi) (\omega ,t, u)$$ is exactly $\nu ^{\flat}$.
\end{proof}

\begin{proposition} \label{omegacontraction}
The maps $s = \gamma_{(\Sigma \Omega)^nX} : \Omega'(\Sigma \Omega)^nX \to \Omega'(\Sigma \Omega)^{n+1}X$ define a contraction of
the facial resolution $\Omega'\Lambda_{\bullet}X \to \Omega 'X$.
\end{proposition}

\begin{proof}
We have $(\Omega '(\ev_X)\circ \gamma_X)(\omega ,t) = \Omega '(\ev_X)(\nu (\omega ,t),t) = (\beta (\omega ,t),t)$ where
$$\beta (\omega ,t)(u) = \left\{ \begin{array}{ll}
\omega _t(\frac{u}{t}) = \omega (u), &  u < t,\\
\ast = \omega (u), & u \geq t.
\end{array}\right.$$
Hence $(\Omega '(\ev_X)\circ \gamma_X) = \id_{\Omega'X}$.

 In the same way one has $(\Omega '(\ev_{(\Sigma \Omega)^nX})\circ \gamma_{(\Sigma \Omega)^nX}) = \id_{(\Sigma \Omega)^nX}$. This shows the relation $d_0 \circ s = \id$. It remains to check that $d_j \circ s = s\circ d_{j-1}$, for $j \geq 1$. For $(\omega ,t) \in \Omega '(\Sigma \Omega)^{n}X$ we have $(d_j \circ s)(\omega ,t) = (\Omega'(\Sigma \Omega)^j(\ev_{(\Sigma \Omega)^{n-j}X})\circ \gamma_{(\Sigma \Omega )^nX})(\omega ,t) = (\sigma (\omega ,t),t)$ where
$$\sigma (\omega ,t)(u) = \left\{ \begin{array}{ll}
(\Sigma \Omega)^j(\ev_{(\Sigma \Omega)^{n-j}X})\left[\omega _t, \frac{u}{t}\right] = \left[(\Sigma \Omega)^{j-1}(\ev_{(\Sigma \Omega)^{n-j}X})\circ \omega _t, \frac{u}{t}\right], &  \!\!\!\!u < t,\\
(\Sigma \Omega)^j(\ev_{(\Sigma \Omega)^{n-j}X})\left[c_*,0\right] = \left[c_*,0\right], &
\!\!\!\! u \geq t.
\end{array}\right.$$
On the other hand, $(s\circ d_{j-1})(\omega ,t) = (\gamma_{(\Sigma \Omega )^{n-1}X}\circ \Omega'(\Sigma \Omega)^{j-1}(\ev_{(\Sigma \Omega)^{n-j}X}))(\omega ,t) = (\tau (\omega ,t),t)$
where
$$\tau (\omega ,t)(u) = \left\{ \begin{array}{ll}
\left[((\Sigma \Omega)^{j-1}(\ev_{(\Sigma \Omega)^{n-j}X})\circ \omega )_t, \frac{u}{t}\right], &  u < t,\\
\left[c_*,0\right], & u \geq t.
\end{array}\right.$$
This shows that $d_j \circ s = s\circ d_{j-1}$ $(j \geq 1)$.
\end{proof}

\section{Second part of \thmref{thm:main}: the map $G_n(X)\to  \|\Lambda_{\bullet}X\|_n$}

\noindent A \emph{bifacial space} is a facial object in the
category $d{\bf Top}$ of facial spaces. We will use notations like
$Z_{\bullet}^{\bullet}$ to denote bifacial spaces and refer to the
upper index as the column index and to the lower index as the row 
index. In this way, a bifacial space can be represented by a
diagram of the following type:

$$\xymatrix{
\vdots \ar@<4pt>[d]^-{\del_{n+1}}  \ar@{}[d]|-{..}\ar@<-4pt>[d]_-{\del_0}
&
\vdots \ar@<4pt>[d]^-{\del_{n+1}}  \ar@{}[d]|-{..}\ar@<-4pt>[d]_-{\del_0}
&
\vdots \ar@<4pt>[d]^-{\del_{n+1}}  \ar@{}[d]|-{..}\ar@<-4pt>[d]_-{\del_0}
&&
\vdots \ar@<4pt>[d]^-{\del_{n+1}}  \ar@{}[d]|-{..}\ar@<-4pt>[d]_-{\del_0}
&
\vdots \ar@<4pt>[d]^-{\del_{n+1}}  \ar@{}[d]|-{..}\ar@<-4pt>[d]_-{\del_0}
&\\
 Z_{n}^{0 }
    \ar@<4pt>[d]^-{\del_n}  \ar@{}[d]|-{..}\ar@<-4pt>[d]_-{\del_0}
        &\ar@<2pt>[l]^{d_1} \ar@<-2pt>[l]_{d_0} Z_{n}^{1}
        \ar@<4pt>[d]^-{\del_n}  \ar@{}[d]|-{..}\ar@<-4pt>[d]_-{\del_0}
             &\ar@<4pt>[l]^{d_2} \ar[l]|{d_1}\ar@<-4pt>[l]_{d_0} Z_{n}^{2}
             \ar@<4pt>[d]^-{\del_n} \ar@{}[d]|-{..}\ar@<-4pt>[d]_-{\del_0}
                   &\cdots
                      &Z_{n}^{p-1}\ar@<4pt>[d]^-{\del_n}  \ar@{}[d]|-{..}\ar@<-4pt>[d]_-{\del_0}
                           &\ar@<4pt>[l]^-{d_p}  \ar@{}[l]|-{:}\ar@<-5pt>[l]_-{d_0}Z_{n}^{p}
                           \ar@<4pt>[d]^-{\del_n}  \ar@{}[d]|-{..}\ar@<-4pt>[d]_-{\del_0}
                            &\cdots\\
\vdots \ar@<4pt>[d]^-{\del_2}  \ar[d]\ar@<-4pt>[d]_-{\del_0}
        &\vdots \ar@<4pt>[d]^-{\del_2}  \ar[d]\ar@<-4pt>[d]_-{\del_0}
             &\vdots \ar@<4pt>[d]^-{\del_2}  \ar[d]\ar@<-4pt>[d]_-{\del_0}
                  &\cdots
                     &\vdots \ar@<4pt>[d]^-{\del_2}  \ar[d]\ar@<-4pt>[d]_-{\del_0}
                           &\vdots \ar@<4pt>[d]^-{\del_2}  \ar[d]\ar@<-4pt>[d]_-{\del_0}
                             &\cdots\\
Z_{1}^{0 }
    \ar@<2pt>[d]^-{\del_1} \ar@<-2pt>[d]_-{\del_0}
        &\ar@<2pt>[l]^{d_1} \ar@<-2pt>[l]_{d_0}Z_{1}^{1}
        \ar@<2pt>[d]^-{\del_1} \ar@<-2pt>[d]_-{\del_0}
            &\ar@<4pt>[l]^{d_2} \ar[l]|{d_1}\ar@<-4pt>[l]_{d_0}Z_{1}^{2}
            \ar@<2pt>[d]^-{\del_1} \ar@<-2pt>[d]_-{\del_0}
                 &\cdots
                     &Z_{1}^{p-1}\ar@<2pt>[d]^-{\del_1} \ar@<-2pt>[d]_-{\del_0}
                           &\ar@<4pt>[l]^-{d_p}  \ar@{}[l]|-{:}\ar@<-5pt>[l]_-{d_0}Z_{1}^{p}
                           \ar@<2pt>[d]^-{\del_1} \ar@<-2pt>[d]_-{\del_0}
                            &\cdots\\
Z_{0}^{0 }
        &\ar@<2pt>[l]^{d_1} \ar@<-2pt>[l]_{d_0}Z_{0}^{1}
            &\ar@<4pt>[l]^{d_2} \ar[l]|{d_1}\ar@<-4pt>[l]_{d_0}Z_{0}^{2}
                 &\cdots
                     &Z_{0}^{p-1}
                           &\ar@<4pt>[l]^-{d_p}  \ar@{}[l]|-{:}\ar@<-5pt>[l]_-{d_0}Z_{0}^{p}
                           &\cdots
}$$
As in this diagram we shall reserve the notation $\del_i$
for the face operators of a column facial space and the notation
$d_i$ for the face operators of a row facial space. For any $k$,
$|Z^k_{\bullet}|_{m}$ (resp. $|Z_k^{\bullet}|^{m}$) is the
realization up to $m$ of the $k$th column (resp. $k$th row)
and $|Z^{\bullet}_{\bullet}|_{m}$ (resp.
$|Z_{\bullet}^{\bullet}|^{m}$) is the facial space obtained by 
realizing each column (resp. each row) up to $m$.\\

The construction of the map $G_n(X)\to  \|\Lambda_{\bullet}X\|_n$ relies heavily on the following result which is analogous to a theorem of A. Libman \cite{LibmanII}. As A. Libman has pointed out to the authors, this result can be derived from \cite{LibmanII} (private communication). For the convenience of the reader, we include, in an appendix, an independent proof of the particular case we need.

\begin{theorem}\label{Libman}
Consider a facial space $Z_{\bullet}^{-1 }$ and a facial resolution
$Z_{\bullet}^{\bullet} \stackrel{d_0}{\to} Z_{\bullet}^{-1}$ such
that each row $Z_{k}^{\bullet} \stackrel{d_0}{\to} Z_{k}^{-1}$
admits a contraction. Then, for any $n$, there exists a not
necessarily base-point preserving continuous map
$|Z_{\bullet}^{-1}|_n \to ||Z_{\bullet}^{\bullet}|^n|_n$ which is a
section up to free homotopy of the canonical map
$||Z_{\bullet}^{\bullet}|_n|^n\to |Z_{\bullet}^{-1}|_n$.
\end{theorem}

The second part of \thmref{thm:main} can be stated as follows.

\begin{theorem}\label{thm:hardway}
Let $X\in \Top$ be a simply connected space. For each $n \in \N$ the $n$th Ganea map $G_n(X) \to X$ factors up to (pointed) homotopy through the canonical map $\|\Lambda _{\bullet}X\|_n \to X$.
\end{theorem}

\begin{proof}
Consider the column facial space $Z_{\bullet}^{-1} = {\mathcal G}_{\bullet}(X)$ and the facial resolution $Z_{\bullet}^{-1} \leftarrow Z_{\bullet}^{\bullet} $ where $Z_{i}^{j} = {\mathcal G}_{i}(\Lambda_{j}X)$. Each row facial resolution $$Z_{i}^{-1} = 
{\mathcal G}_{i}(X) \leftarrow Z_{i}^{\bullet} = {\mathcal G}_{i}(\Lambda_{\bullet}X)$$ admits a contraction. Since  ${\mathcal G}_{0}(\Lambda_{\bullet}X) = \ast$, this is clear for $i = 0$. For $i > 0$,
${\mathcal G}_{i}(\Lambda_{\bullet}X) = (\Omega '\Lambda_{\bullet}X)^i$. Indeed, since, by \propref{omegacontraction}, the facial resolution
$\Omega 'X \leftarrow \Omega'\Lambda_{\bullet}X$ admits a contraction, its $i$th power also admits a contraction. 

For $n \in \N$ consider the commutative diagram
$$\xymatrix{
B_n\Omega 'X\ar[r]^{}_{} \ar[d]^{}_{} &
|{\mathcal G}_{\bullet}(X)|_n \ar[d] &
||{\mathcal G}_{\bullet}(\Lambda_{\bullet}X)|_n|^n \ar[l] \ar[d]
\\
B\Omega 'X\ar[r]^{}_{} &
|{\mathcal G}_{\bullet}(X)|_{\infty} &
||{\mathcal G}_{\bullet}(\Lambda_{\bullet}X)|_{\infty}|^n \ar[l]
} $$
in which the left-hand square is the natural square of \propref{cat}. Recall that the lower left horizontal map is a homotopy equivalence. Since $X$ is simply connected, $X$ is naturally weakly equivalent to $B\Omega 'X$ and hence to $|{\mathcal G}_{\bullet}(X)|_{\infty}$. It follows that the map $||{\mathcal G}_{\bullet}(\Lambda_{\bullet}X)|_{\infty}|^n \to |{\mathcal G}_{\bullet}(X)|_{\infty}$ is weakly equivalent to the map $|\Lambda_{\bullet}X|_n \to X$. Since this last map factors through the map $\|\Lambda_{\bullet}X\|_n \to X$ and since, by \thmref{Libman}, the upper right horizontal map of the diagram above admits a free homotopy section, we obtain a diagram
$$\xymatrix{
B_n\Omega 'X\ar[r]^{}_{} \ar[d]^{}_{}
& \|\Lambda_{\bullet}X\|_n\ar [d]^{}_{}
\\
B\Omega 'X\ar[r]^{f}_{}
& X
}$$
which is commutative up to free homotopy and in which $f$ is a (pointed) homotopy equivalence. Since the left hand vertical map is equivalent to the Ganea map $G_n(X) \to X$, there exists a diagram
$$\xymatrix{
G_n(X) \ar[r]^{}_{} \ar[d]^{}_{}
& \|\Lambda_{\bullet}X\|_n\ar [d]^{}_{}
\\
X\ar[r]^{g}_{}
& X
}$$
which is commutative up to free homotopy and in which $g$ is a (pointed) homotopy equivalence. This implies that the Ganea map $G_n(X) \to X$ factors up to free homotopy through the canonical map $\|\Lambda _{\bullet}X\|_n \to X$. Since $X$ is simply connected and $\|{\Lambda}_{\bullet}X\|_n$ is connected, the Ganea map $G_n(X) \to X$ also factors up to pointed homotopy through the canonical map $\|\Lambda _{\bullet}X\|_n \to X$.
\end{proof}
\section{Proof of \thmref{thm:main2}}
\begin{proof}
Recall the homotopy fiber sequence
$$
\xymatrix @1{\Omega X\ast \Omega X\ar[r]^-h&\Sigma\Omega X\ar[r]^-{d_0}&X
}$$
where $h$ is the Hopf map. This sequence is natural in $X$ and the space $G_2(X)$ is equivalent to the pushout of
$\xymatrix @1{\cC(\Omega X\ast \Omega X)&\Omega X\ast\Omega X\ar[l]\ar[r]&\Sigma\Omega X
}$, where $\cC(Y)$ denotes the (reduced) cone over a space~$Y$.
We use the following diagram
$$\xymatrix@=10pt{
(2)&\cC(\Omega X\ast \Omega X)&&
\cC(\Omega\Sigma\Omega X\ast \Omega\Sigma\Omega X)\ar[ll]_{d_0}&
\cC(\Omega\left(\Sigma\Omega\right)^2X\ast \Omega\left(\Sigma\Omega\right)^2X)
\ar@<2pt>[l]^-{d_1} \ar@<-2pt>[l]_-{d_0}&\\
(1)&\Omega X\ast \Omega X\ar[u]\ar[d]_h&&
\Omega\Sigma\Omega X\ast \Omega\Sigma\Omega X\ar[u]\ar[ll]_{d_0}\ar[d]&
\Omega\left(\Sigma\Omega\right)^2X\ast \Omega\left(\Sigma\Omega\right)^2X\ar[u]\ar[d]
\ar@<2pt>[l]^-{d_1} \ar@<-2pt>[l]_-{d_0}&\\
(0)&\Sigma\Omega X\ar[dd]_(.4){d_0}&&
\left(\Sigma\Omega\right)^2 X\ar[ll]_{d_0}\ar[dd]_(.4){d_0}&
\left(\Sigma\Omega\right)^3 X\ar[dd]_(.4){d_0}
\ar@<2pt>[l]^-{d_1} \ar@<-2pt>[l]_-{d_0}&\\
&&\ar@{--}[uuu]\ar@{--}[rrr]&&&\\
(-1)&X&&\Sigma\Omega X\ar[ll]_{d_0}&
\left(\Sigma\Omega\right)^2X
\ar@<2pt>[l]^-{d_1} \ar@<-2pt>[l]_-{d_0}&\\
%
}$$
We observe that
\begin{itemize}
\item the image of Line (-1) by $\Omega$ has a  contraction in the obvious sense;
\item Line (0) is the image of Line (-1) by $\Sigma\Omega$ therefore Line (0) admits a contraction;
\item the face operators of Line (1) are the maps $\Omega d_i\ast\Omega d_i$ with the face operators $d_i$ of Line (-1), thus Line (1) admits a contraction;
\item Line (2) admits a contraction induced by the previous one.
\end{itemize}
From the expression of the Hopf map $h\colon \Omega X\ast\Omega X\to \Sigma\Omega X$, $h([\alpha, t,\beta])=[\alpha^{-1}\beta,t]$, we observe that the map
$H\colon(\Omega X\ast\Omega X)\times [0,1]\to X$, defined by
$H([\alpha,t,\beta],s)=\alpha^{-1}\beta(st)$, induces a natural extension of $d_0\circ h$ to 
$\cC(\Omega X\ast \Omega X)$. Therefore, we can  complete the diagram by maps from Line (2) to Line (-1) which are compatible with face operators.\\
Denote by $\tilde{G}$ the homotopy colimit of the framed part of the diagram. We have a commutative square:
$$\xymatrix{
G_2(X)\ar[d]&\tilde{G}\ar[l]\ar[d]\\
X&\|\Lambda_\bullet X\|_1\ar[l]
}$$
\lemref{lem:petitlibman} provides a homotopy section of the map
$\tilde{G}\to G_2(X)$. Thus we obtain a map 
$$G_2(X)\to \|\Lambda_\bullet X\|_1$$
up to homotopy over $X$.
\end{proof}

\begin{lemma}\label{lem:petitlibman}
We consider the following diagram in $\Top$, satisfying $d_0\circ d_0=d_0\circ d_1$ and the obvious commutativity conditions.$$
\xymatrix{
&\ar@{--}[rrr]\ar@{--}[dddd]
&&&\ar@{--}[dddd]\\
A_{-1}&&
A_0\ar[ll]_(.4){d_0}&
A_1
\ar@<2pt>[l]^-{d_1} \ar@<-2pt>[l]_-{d_0}&\\
B_{-1}\ar[u]_{\alpha_{-1}}\ar[d]^{\beta_{-1}}&&
B_0\ar[ll]_(.4){d_0}\ar[u]_{\alpha_{0}}\ar[d]^{\beta_{0}}&
B_1\ar[u]_{\alpha_{1}}\ar[d]^{\beta_{1}}
\ar@<2pt>[l]^-{d_1} \ar@<-2pt>[l]_-{d_0}&\\
C_{-1}&&
C_0\ar[ll]_(.4){d_0}&
C_1
\ar@<2pt>[l]^-{d_1} \ar@<-2pt>[l]_-{d_0}&\\
&\ar@{--}[rrr]&&&
}$$
Let $\tilde{G}$ be the homotopy  colimit of the framed part and $G_{-1}$ be the homotopy colimit of the first column. 
We denote by  $\tilde{d}\colon \tilde{G}\to G_{-1}$ the map induced by $d_0$.
If the lines of the previous diagram admit contractions in the obvious sense, then the map $\tilde{d}$ has a (pointed) homotopy section.
\end{lemma}

\begin{proof}
This is a special case of a dual of a result of Libman in \cite{LibmanII}. It is not covered by the proof of the last section but this  situation is simple and we furnish an ad-hoc argument for it.

First we construct maps
$f\colon A_{-1}\to \|A_\bullet  \|_1$,
$g\colon B_{-1}\to  \|B_\bullet  \|_1$
and
$k\colon C_{-1}\to  \|C_\bullet \|_1$
such that
$ \|\alpha_\bullet \|_1\circ g\simeq f\circ \alpha_{-1}\text{ and }
k\circ \beta_{-1}\simeq  \|\beta_\bullet \|_1\circ g$.
With the same techniques as in \propref{quotient}, it is clear that $\|A_\bullet\|_1$ is homeomorphic to the quotient $A\rtimes \Delta^1$ by the relation $(a,t_0,t_1)\sim (sd_ia,0,1)$ if $t_i=0$.  So, we define $f$, $g$  and $k$ by
$$f(a)=[s_As_A(a),0,1], g(b)=[s_Bs_B(b),0,1]\text{ and }
k(c)=[s_Cs_C(c),0,1].$$
A computation gives:
\begin{eqnarray*}
 \|\alpha_\bullet \|_1\circ g(b)&=&[\alpha_1s_Bs_B(b),0,1]\\
&=& [s_Ad_0\alpha_1s_Bs_B(b),0,1]\\
&=&[s_A\alpha_0d_0s_Bs_B(b),0,1]\\
&=& [s_A\alpha_0s_B(b),0,1]\\
f\circ \alpha_1(b)&=&
[s_As_A\alpha_{-1}(b),0,1]\\
&=&
[s_As_Ad_0\alpha_0s_B(b),0,1]\\
&=&[s_Ad_1s_A\alpha_0s_B(b),0,1]\\
&=&[s_A\alpha_0s_B(b),1,0],
\end{eqnarray*}
the last equality coming from our construction of $ \|A_\bullet \|_1$.
These two points,
$ \|\alpha_\bullet \|_1\circ g(b)$ and $f\circ \alpha_1(b)$,
 are canonically joined by a path that reduces to a point if $b=*$. The same argument gives the similar result for $k$.  We observe now that these homotopies give a map between the two mapping cylinders which is a section up to pointed homotopy.
\end{proof}
\section{Open questions}

The main open question after these results concerns  the existence of maps over $X$ up to homotopy,
$G_n(X)\to \|\Lambda_\bullet X\|_{n-1}$ 
for any $n$. This question is related to the Lusternik-Schnirelman category (LS-category in short) $\cat X$ of a topological space $X$. Recall that $\cat X\leq n$ if and only if the Ganea fibration $G_n(X)\to X$ admits a section. 
The truncated resolutions bring a new homotopy invariant $\ell_{\Sigma\Omega}(X)$ defined in a similar way as follows:
$$\ell_{\Sigma\Omega}(X)\leq n \text{ if the map } \|\Lambda_\bullet X\|_{n-1}\to X \text{ admits a homotopical section.}$$
From \thmref{thm:main} and  \thmref{thm:main2}, we know that this new invariant coincides with the LS-category for spaces of LS-category less than or equal to 2 and satisfies
$$\cat X \leq \ell_{\Sigma\Omega}(X)\leq 1+ \cat X.$$
Grants to the result in dimension 2, $ \ell_{\Sigma\Omega}(X)$  does not coincide with the cone length. We conjecture its equality with the LS-category and the existence of maps $G_n(X)\to \|\Lambda_\bullet X\|_{n-1}$ over $X$ up to homotopy.

\medskip
We now extend our study by considering a cotriple $T$.
 Recall that a cotriple $(T,\eta,\varepsilon)$ on
$\Top$ is a functor
$T :
\Top\rightarrow
\Top$ together with two natural transformations
$\eta_X\colon T(X)\rightarrow X$ and $\varepsilon_X\colon T(X)\rightarrow T^2(X)$ such that:\\
\centerline{$\varepsilon_{F(X)}\circ \varepsilon_X=F(\varepsilon_X)\circ \varepsilon_X$
and
$\eta_{T(X)}\circ \varepsilon_X=T(\eta_X)\circ \varepsilon_X=\id_{T(X)}$.}

It is well known that $T$ gives a simplicial space $\Lambda^T_\bullet X$ defined by $\Lambda^T_n X =T^{n+1}(X)$. From it, we deduce a facial space and the truncated realizations $\|\Lambda^T_\bullet X\|_n$. If $T$ satisfies $T(*)\sim *$, takes its values in suspensions and $\Omega'(\Lambda^T_\bullet X)$ admits a contraction, a careful reading of the proofs in this work shows that we get the same conclusions as in \thmref{thm:main} and  \thmref{thm:main2}  with the Ganea spaces $G_n(X)$ and the realizations $\|\Lambda^T_\bullet X\|_i$.

\medskip
We could also use a construction of the Ganea spaces adapted to the cotriple $T$ as follows.

\begin{definition}\label{defi:TGanea}
Let $T$ be a cotriple and $X$ be a space, the \emph{$n$th fibration of Ganea associated
to
$T$ and $X$} is defined inductively by:

-- $p_1^T\colon G_1^T(X)\rightarrow X$ is the associated fibration to $\eta_X\colon T(X)
\rightarrow X$,

-- if $p^T_n\colon G_n^T(X)\rightarrow X$ is defined, we denote by $F_n^T(X)$ its homotopy
fiber and build a map ${p'}_{n+1}^T\colon G_n^T(X)\cup {\cC}(T(F_n^T(X))\rightarrow X$ as
$p^T_n$ on $G_n^T(X)$ and sending the cone $\cC(T(F_n^T(X))$ on the base point. The
fibration
$p_{n+1}^T$ is the associated fibration to ${p'}_{n+1}^T$.
\end{definition}

The results of this paper and the questions above have their analog in this setting. 
New approximations of spaces arise from the truncated realizations $\|\Lambda^T_\bullet X\|_i$ and from the adapted fiber-cofiber constructions. One natural problem is to look for a comparison between them. These questions  can also be stated in terms of LS-category. For instance, does the Stover resolution (see \cite{Sto90}) of a space by wedges of spheres give the $s$-category defined in \cite{Sc-Ta99b}?

\section{Appendix: Proof of \thmref{Libman}}

The purpose of this appendix is to give a proof of \thmref{Libman}.
This proof is contained in the \subsecref{proof} below and uses the
constructions and notation of the following subsection.

\subsection{$n$-facial spaces and $n$-rectifiable maps}

\noindent Let $n\geq 0$ be an integer. A facial space $X_{\bullet}$
is a \emph{$n$-facial space} if, for any $k\geq n+1$, $X_k=*$.
To any facial space $Y_{\bullet}$, we can associate an $n$-facial
space $T^n_{\bullet}(Y)$ by setting $T^n_{k}(Y)=Y_k$ if $k\leq n$
and $T^n_{k}(Y)=*$ if $k\geq n+1$. Obviously, for any $k\leq n$, we
have
$|T^n_{\bullet}(Y)|_k=|Y_{\bullet}|_k$.\\

 \noindent Let $Y_{\bullet}$ be a facial
space with face operators $\del_i:Y_k\to Y_{k-1}$. We associate
to $Y_{\bullet}$ two $n$-facial spaces $I^n_{\bullet}(Y)$ and
$J^n_{\bullet}(Y)$ and morphisms $\eta,\zeta,\pi,\overline{\pi}$
which induce homotopy equivalences between the realizations up
to $n$ and such that the following diagram is commutative:
$$\xymatrix{
T^n_{\bullet}(Y)
\ar[r]^{\eta}\ar[rd]_{\id}&I^n_{\bullet}(Y)\ar[d]_{{\pi}} &
J^n_{\bullet}(Y) \ar[l]_{\zeta} \ar[ld]^{\overline{\pi}}\\
&T^n_{\bullet}(Y).& }$$ For any integer $k\geq 1$ we denote by
$\del_{\underline{k}}$ the set $\{\del_0,...,\del_k\}$ of the $(k+1)$
face operators $\del_i:Y_k\to Y_{k-1}$ and, for any integer $l\geq
k$, we set $\del_{\underline{k}\,:\,\underline{l}}:=
\del_{\underline{k}}\times \del_{\underline{k+1}}\times...\times \del_{\underline{l}}$. \\

\paragraph{\bf The $n$-facial space $J^n_{\bullet}(Y)$.}
For $0\leq k\leq n$, consider the space:
$$\left(Y_k\times \Delta^0\right) \coprod \coprod_{m=1}^{n-k}\left(\del_{\underline{k+1}\,:\,\underline{k+m}}\times Y_{k+m}\times \Delta^m\right).$$
An element of this space
will be written $(\del_{i_1},...,\del_{i_m},y,t_0,...,t_m)$ with the convention
$(\del_{i_1},...,\del_{i_m},y,t_0,...,t_m)=(y,1)$ if $m=0$. Set
$$J^n_{k}(Y):=\left(\left(Y_k\times \Delta^0\right) \coprod \coprod_{m=1}^{n-k}\left(\del_{\underline{k+1}\,:\,\underline{k+m}}\times Y_{k+m}\times \Delta^m\right)\right)/\sim$$
where the relations are given by
$$
(\del_{i_1},...,\del_{i_m},y,t_0,...,t_m)\sim (\del_{i_1},...,\del_{i_{m-1}},\del_{i_m}y,t_0,...,t_{m-1}), \quad \mbox{if } t_m=0,$$
and
$$
(\del_{i_1},...,\del_{i_p},\del_{i_{p+1}},...\del_{i_m},y,t_0,...,t_m)\sim
(\del_{i_1},...,\del_{i_{p+1}-1},\del_{i_{p}},...\del_{i_m},y,t_0,...,t_m),
$$
if $t_p=0$ and $i_p<i_{p+1}$.

Together with the face operators $J{\del}_i:J^n_{k}(Y)\to
J^n_{k-1}(Y)$, $0\leq i\leq k$, defined by
$$J{\del}_i(\del_{i_1},...,\del_{i_m},y,t_0,...,t_m)=(\del_i,\del_{i_1},...,\del_{i_m},y,0,t_0,...,t_m),$$
$J^n_{\bullet}(Y)$ is a $n$-facial space.\\

\paragraph{\bf The $n$-facial space $I^n_{\bullet}(Y)$.} For $0\leq k\leq n$, we consider now the space:
$$\left(Y_k\times \Delta^1\right) \coprod \coprod_{m=1}^{n-k}\left(\del_{\underline{k+1}\,:\,\underline{k+m}}\times Y_{k+m}\times \Delta^{m+1}\right).$$
We write $(\del_{i_1},...,\del_{i_m},y,t_0,...,t_{m+1})$ the
elements of that space with the convention
$(\del_{i_1},...,\del_{i_m},y,t_0,...,t_{m+1})=(y,t_0,t_1)$ if
$m=0$. The space $I^n_k(Y)$ is defined to be the quotient
$$I^n_k(Y):=\left(\left(Y_k\times \Delta^1\right) \coprod \coprod_{m=1}^{n-k}\left(\del_{\underline{k+1}\,:\,\underline{k+m}}
\times Y_{k+m}\times \Delta^{m+1}\right)\right)/\sim$$ with respect to the
relations
$$
(\del_{i_1},...,\del_{i_m},y,t_0,...,t_{m+1})\sim (\del_{i_1},...,\del_{i_{m-1}},\del_{i_m}y,t_0,...,t_m), \quad \mbox{if } t_{m+1}=0,$$
and
$$
(\del_{i_1},...,\del_{i_p},\del_{i_{p+1}},...\del_{i_m},y,t_0,...,t_{m+1})\sim
(\del_{i_1},...,\del_{i_{p+1}-1},\del_{i_{p}},...\del_{i_m},y,t_0,...,t_{m+1}),
$$
if $t_{p+1}=0$ and $i_p<i_{p+1}$.

Together with the face operators $I{\del}_i:I^n_k(Y)\to
I^n_{k-1}(Y)$, $0\leq i\leq k$, defined by
$$I{\del}_i(\del_{i_1},...,\del_{i_m},y,t_0,t_1,...,t_{m+1})=(\del_i,\del_{i_1},...,\del_{i_m},y,t_0,0,t_1,...,t_{m+1}),$$
$I^n_{\bullet}(Y)$ is a $n$-facial space.\\

\paragraph{\bf The morphisms $\eta,\zeta,\pi,\overline{\pi}$}
The facial maps $\eta: T^n_{\bullet}(Y)\to I^n_{\bullet}(Y)$,
$\zeta: J^n_{\bullet}(Y)\to I^n_{\bullet}(Y)$, $\pi:I^n_{\bullet}(Y)
\to T^n_{\bullet}(Y)$ and $\overline{\pi}:J^n_{\bullet}(Y)\to
T^n_{\bullet}(Y)$ are respectively defined (for $k\leq n$) by:
$$\begin{array}{l}
\eta_k(y)=(y,1,0),\\[.2cm]
\zeta_k(\del_{i_1},...,\del_{i_m},y,t_0,...,t_m)=(\del_{i_1},...,\del_{i_m},y,0,t_0,...,t_m),\\[.2cm]
\pi_k(\del_{i_1},...,\del_{i_m},y,t_0,...,t_{m+1})=\del_{i_1}\cdots\del_{i_m}y \quad \mbox{and} \quad \pi_k(y,t_0,t_1)=y,\\[.2cm]
\overline{\pi}_k=\pi_k\circ \zeta_k.
\end{array}$$

We have $\pi_k\circ \eta_k=\id$ so that the following diagram is commutative:
$$\xymatrix{
T^n_{\bullet}(Y)
\ar[r]^{\eta}\ar[rd]_{\id}&I^n_{\bullet}(Y)\ar[d]_{{\pi}} &
J^n_{\bullet}(Y) \ar[l]_{\zeta} \ar[ld]^{\overline{\pi}}\\
&T^n_{\bullet}(Y).& }$$ In order to see that these morphisms induce
homotopy equivalences between the realizations up to $n$, it
suffices to see that, for any $k$, $0\leq k\leq n$, the maps
$\eta_k,\zeta_k,\pi_k,\overline{\pi}_k$ are homotopy equivalences.
Thanks to the commutativity of the diagram above we just have to
check it for the maps $\pi_k$ and $\overline{\pi}_k$. These two maps
admit a section: we have already seen that $\pi_k\circ \eta_k=\id$
and, on the other hand, the map $\varphi_k: T^n_{k}(Y) \to
J^n_{k}(Y)$ given by $\varphi_k(y)=(y,1)$ (which does not commute
with the face operators) satisfies $\overline{\pi}_k\circ
\varphi_k=\id$. The conclusion follows then from the fact that the
two homotopies
$$\begin{array}{rcl}
H_k:I^n_k(Y)\times I&\to &I^n_k(Y)\\
((\del_{i_1},...,\del_{i_m},y,t_0,...,t_{m+1}),u)&\mapsto&
(\del_{i_1},...,\del_{i_m},y,u+(1-u)t_0),\\
&&\hskip 2.7cm (1-u)t_1,..,(1-u)t_{m+1})
\\[.2cm]
\overline{H}_k:J^n_{k}(Y)\times I&\to &J^n_{k}(Y)\\
((\del_{i_1},...,\del_{i_m},y,t_0,...,t_{m}),u)&\mapsto&
(\del_{i_1},...,\del_{i_m},y,u+(1-u)t_0,\\
&&\hskip 2.7cm (1-u)t_1,..,(1-u)t_{m})
\end{array}$$
satisfy $H_k(-,0)=\id$, $H_k(-,1)=\eta_k\circ \pi_k$ and $\overline{H}_k(-,0)=\id$,
$\overline{H}_k(-,1)=\varphi_k\circ \overline{\pi}_k$.\\

\paragraph{\bf $n$-rectifiable map.}
We write $\varphi:T^n_{\bullet}(Y)\dasharrow J^n_{\bullet}(Y)$ to
denote the collection of maps $\varphi_k: T^n_k(Y) \to J^n_k(Y)$
given by $\varphi_k(y)=(y,1)$. Recall that $\varphi$ is not a
morphism of facial spaces since it does not satisfy the usual rules of commutation with the face
operators. In the same way we write $\psi:Y_{\bullet}\dasharrow
Z_{\bullet}$ for a collection of maps $\psi_k:Y_{k}\dasharrow Z_{k}$
which do not satisfy the usual rules 
 of commutation  with the face operators and we say that $\psi$
is an \emph{$n$-rectifiable map} if there exists a morphism of
facial spaces $\overline{\psi}:J^n_{\bullet}(Y)\to T^n_{\bullet}(Z)$
such that $\overline{\psi}_k\circ\varphi_k=\psi_k$ for any $k\leq
n$. So, an $n$-rectifiable map $\psi:Y_{\bullet}\dasharrow
Z_{\bullet}$ induces a map between the realizations up to $n$ of the
facial spaces $Y_{\bullet}$ and $Z_{\bullet}$.\\

\subsection{Proof of \thmref{Libman}}\label{proof}

Let $Z_{\bullet}^{\bullet} \stackrel{d_0}{\to} Z_{\bullet}^{-1}$ be
a facial resolution of a facial space $Z_{\bullet}^{-1}$ such that
each row $Z_{k}^{\bullet} \stackrel{d_0}{\to} Z_{k}^{-1}$ admits a
contraction and let $n\geq 0$. We first note that the realization of
$Z_{\bullet}^{\bullet}$ up to $p$ along the rows and up to $n$ along
the columns leads to two canonical maps:
$$||Z_{\bullet}^{\bullet}|^p|_ n\to |Z_{\bullet}^{-1}|_n \qquad ||Z_{\bullet}^{\bullet}|_n|^ p\to |Z_{\bullet}^{-1}|_n.$$
Induction on $p$ and standard colimit arguments show that these two
maps are equal (up to homeomorphism). Here we  prove that
$||Z_{\bullet}^{\bullet}|^p|_ n\to |Z_{\bullet}^{-1}|_n$ admits a
homotopy section.\\

 \noindent For any $k$, we denote by $s_k$ the contraction of the $k$th row
$$\xymatrix{
Z_{k}^{-1 } & \ar[l]_{d_0} Z_{k}^{0 } &\ar@<2pt>[l]^{d_1}
\ar@<-2pt>[l]_{d_0}Z_{k}^{1} &\ar@<4pt>[l]^{d_2}
\ar[l]|{d_1}\ar@<-4pt>[l]_{d_0}X_{k}^{2} &\cdots &Z_{k}^{n-1}
&\ar@<4pt>[l]^-{d_n}  \ar@{}[l]|-{:}\ar@<-4pt>[l]_-{d_0}Z_{k}^{n}
}$$ and, in order to simplify the notation we will write $L_k$ for
the realization up to $n$ of this facial space. That is,
$L_k=|Z_k^{\bullet}|^n$. Recall, from \propref{quotient}, that the
existence of the contraction permits the following description of
$L_k$:
 $$L_k=Z_k^n\times \Delta^n /\sim$$
 where the relation is given by
 $$(z,t_0,...,t_i,...,t_n)\sim (s_kd_iz,0,t_0,...,\hat{t}_i,...,t_n)\quad \mbox{if  } t_i=0.$$
With respect to this description, the canonical map $L_k\to
Z_k^{-1}$ is given by\\
$[z,t_0,...,t_i,...,t_n]\mapsto d_0^{n+1}z$ and is denoted by $\varepsilon_n$ (without reference to $k$).\\

\noindent Realizing all the lines, we obtain a facial map:
$$\xymatrix{
\vdots \ar@<4pt>[d]^-{\del_{n+1}}  \ar@{}[d]|-{..}\ar@<-4pt>[d]_-{\del_0}&
\vdots \ar@<4pt>[d]^-{\del_{n+1}}  \ar@{}[d]|-{..}\ar@<-4pt>[d]_-{\del_0}
\\
 Z_n^{-1}\ar@<4pt>[d]^-{\del_n}
\ar@{}[d]|-{..}\ar@<-4pt>[d]_-{\del_0} & L_n\ar@<4pt>[d]^-{\del_n}
\ar@{}[d]|-{..}\ar@<-4pt>[d]_-{\del_0}  \ar[l]_{\varepsilon_n}
\\
\vdots \ar@<4pt>[d]^-{\del_2} \ar[d] \ar@<-4pt>[d]_-{\del_0}
&\vdots \ar@<4pt>[d]^-{\del_2}  \ar[d]\ar@<-4pt>[d]_-{\del_0} \\
Z_1^{-1}\ar@<2pt>[d]^-{\del_1} \ar@<-2pt>[d]_-{\del_0}
&L_1\ar@<2pt>[d]^-{\del_1} \ar@<-2pt>[d]_-{\del_0}  \ar[l]_{\varepsilon_n}\\
Z_0^{-1} &L_0 \ar[l]_{\varepsilon_n} }$$

\noindent The face operators $\del_i:L_k\to L_{k-1}$ are given by
$\del_i[z,t_0,...,t_n]=[\del_i z,t_0,...,t_n]$. Our aim is thus to
see that the map obtained after realization (and always denoted by
$\varepsilon_n$)
$$\xymatrix{
|Z_{\bullet}^{-1}|_n& |L_{\bullet}|_n   \ar[l]_{\varepsilon_n} }$$
admits a section up to homotopy.

For each $k$, the map $\varepsilon_n:L_k\to Z_k^{-1}$ admits a
(strict) section given by $z\mapsto [s_k^{n+1}z,0,0,...,0,1]$ which
we denote by $\psi_k$. The collection $\psi$ of these maps does not
define a facial map since the contraction $s_k$ are not required to
commute with the face operators $\del_i$ of the columns. The key is that
$\psi:Z_{\bullet}^{-1}\dasharrow L_{\bullet}$ is an $n$-rectifiable
map. We can indeed consider, for each $k\leq n$, the (well-defined)
map $\overline{\psi}_k:J_k^n({Z^{-1}})\to L_k$ given by:
$$\overline{\psi}_k(\del_{i_1},...,\del_{i_m},z,t_0,...,t_m)=
[s_k^{n+1-m}\del_{i_1}s_{k+1}\del_{i_2}s_{k+2}...\del_{i_m}s_{k+m}
z,0,...,0,t_0,...,t_m].$$ Straightforward calculation shows that the
maps $\overline{\psi}_k$ commute with the face operators $\del_i$ so
that the collection $\overline{\psi}$ is a facial map. This morphism
also satisfies $\overline{\psi}_k\circ \varphi_k=\psi_k$ for any
$k\leq n$ (which implies that $\psi$ is an $n$-rectifiable map) and
$\varepsilon_n\overline{\psi}=\overline{\pi}$. We have hence the
following commutative diagram:

$$\xymatrix{
T^n_{\bullet}(Z^{-1})
\ar[r]^{\eta}\ar[rd]_{\id}&I^n_{\bullet}(Z^{-1})\ar[d]_{{\pi}} &
J^n_{\bullet}(Z^{-1}) \ar[l]_{\zeta} \ar[ld]_{\overline{\pi}}\ar[r]^{\overline{\psi}}&T^n_{\bullet}(L)\ar[lld]^{\varepsilon_n}\\
&T^n_{\bullet}(Z^{-1}).& }$$

Since the morphisms $\eta$, $\zeta$, $\pi$ and $\overline{\pi}$
induce homotopy equivalence between the realizations up to $n$, we
get the following situation after realization:
$$\xymatrix{
|T^n_{\bullet}(Z^{-1})|_n
\ar[r]^{\sim}\ar[rd]_{\id}&|I^n_{\bullet}(Z^{-1})|_n\ar[d]_{{\sim}} &
|J^n_{\bullet}(Z^{-1})|_n \ar[l]_{\sim} \ar[ld]_{\sim}\ar[r]^{\overline{\psi}}&|T^n_{\bullet}(L)|_n\ar[lld]^{\varepsilon_n}\\
&|T^n_{\bullet}(Z^{-1})|_n.& }$$ Since
$|T^n_{\bullet}(Z^{-1})|_n=|Z_{\bullet}^{-1}|_n$ and
$|T^n_{\bullet}(L)|_n=|L_{\bullet}|_n$, we obtain that the map
$|L_{\bullet}|_n\to|Z_{\bullet}^{-1}|_n $ admits a homotopy section.
\hfill$\Box$


\bibliographystyle{abbrv}
\bibliography{newbiblio}

\end{document}